\documentclass[11pt,reqno]{amsart}
\usepackage{amscd,amsmath,amssymb,amsthm,amsfonts}
\usepackage{mathrsfs}
\usepackage{color}
\usepackage{todonotes}

\usepackage[margin=37mm]{geometry}
\newcommand{\R}{\mathbb{R}}
\newcommand{\C}{\mathbb{C}}

\newcommand{\mat}[2][rrrrrrrrrrrrrrrrrrrrrrrrrrrrrrrrrrrrrrr]{\left[ \begin{array}{#1}#2 \\ \end{array} \right]}
\newcommand{\Mat}[2][cccccccccccccccccccccccccccccccccccccc]{\left[ \begin{array}{#1}#2 \\ \end{array} \right]}
\newcommand{\ve}{\varepsilon}

\renewcommand{\Re}{\operatorname{Re}}

\newcommand{\myspace}{\qquad\qquad\qquad}

\newcommand{\cH}{{\mathcal H}}

\newtheorem{theorem}{Theorem}[section]

\newtheorem{proposition}[theorem]{Proposition}
\newtheorem{remark}[theorem]{Remark}

\newtheorem{corollary}[theorem]{Corollary}

\numberwithin{equation}{section}

\makeatletter
\@namedef{subjclassname@2020}{\textup{2020} Mathematics Subject Classification}
\makeatother

\date{}

\begin{document}

% TITLE

\title{Well-posedness and stability for a thermoelectromagnetic system}

\author{Francesca Bucci}
\address{Francesca Bucci, Universit\`a degli Studi di Firenze,
Dipartimento di Matematica e Informatica,
Viale Morgagni 67/A, 50134 Firenze, ITALY
}
\email{francesca.bucci@unifi.it}

\author{Matthias Eller}
\address{Matthias Eller, Georgetown University, 
Department of Mathematics and Statistics,
Georgetown~360, 37th and O Streets NW, Washington DC 20057, USA
}
\email{Matthias.Eller@georgetown.edu}

\author{Nella Rotundo}
\address{Nella Rotundo, Universit\`a degli Studi di Firenze,
Dipartimento di Matematica e Informatica,
Viale Morgagni 67/A, 50134 Firenze, ITALY
}
\email{nella.rotundo@unifi.it}

% MSC2020


% Keywords


\begin{abstract}
In this article we consider a system of coupled partial differential equations that 
mirrors the interconnection between the evolution of an electromagnetic field -- described by the Maxwell's system -- and that of the temperature distribution in a bounded region in three-dimension. A discussion of the mathematical model is provided. We establish the well-posedness of the initial-boundary value problems associated with the thermoelectromagnetic system, in an appropriate functional-analytic framework. 
Then, our investigation and main results pertain to the long-time behavior of the solutions. It is shown that either exponential stability or convergence to stationary solutions
hold true, according as the conductivity is positive definite or semidefinite, respectively;
in the latter case, strong stability is attained in specific topological or analytical settings. This complements and expands earlier results obtained for the uncoupled Maxwell's system.
\end{abstract}

\maketitle

% INTRO

\section{Introduction}
Let $\Omega \subset \R^3$ be open, non-empty, bounded, and connected, with a Lipschitz boundary $\partial \Omega$. 
Consider the system of coupled partial differential equations (PDE)
\begin{equation}\label{e:PDE-system}
 \begin{split}
  \ve e_t -\nabla\times h + a \theta_t + \sigma e &= -j 
  \\
  \mu h_t + \nabla\times e &=0 
  \\ 
  (c+\theta_0 \langle \ve^{-1} a,a\rangle) \theta_t 
  + \theta_0 \langle e_t,a\rangle 
  - \nabla\cdot(\kappa \nabla \theta) &= r
 \end{split}
 \quad \mbox{ in } Q:= \R_+ \times \Omega
\end{equation}
in the unknowns $(e,h)$ (electromagnetic field) and $\theta$ (temperature distribution).
We recall that $\ve$, $\mu$, $\sigma$ denote the electric permittivity, the magnetic permeability and the conductivity, respectively; here 
\begin{equation*}
\ve,\mu,\kappa,\sigma \in L^\infty(\Omega,\R^{3\times 3}_{\mathrm{sym}})\,, 
\; a\in L^\infty(\Omega,\R^3)\,, \; c,\theta_0\in \R\,.
\end{equation*} 
It is assumed that $\theta_0, c>0$, that $\ve,\mu,\kappa$ are uniformly positive definite, while $\sigma$ is allowed to be positive semidefinite.
Here and henceforth $\langle\cdot,\cdot \rangle$ denotes the Euclidean scalar product in
$\mathbb{C}^d$ for some $d>0$.

The PDE model under consideration -- which combines an anisotropic electromagnetic behavior with thermally induced polarization and internal heat generation -- is suitable for describing a range of anisotropic dielectric materials that also exhibit thermoelectric or {\em pyroelectric} effects. 
Several real-world materials that fit this description and for which the constant 
$c>0$ (related to heat absorption) is physically meaningful are highlighted in the Appendix. 
While a simpler version of \eqref{e:PDE-system} is found in Fabrizio and Morro \cite[\S\,7.19.1, p. 374]{FabMor}, the mathematical model is explained in fuller detail in the next section.

System \eqref{e:PDE-system} is augmented by initial conditions
\begin{equation}\label{e:IC}
 e(0,x)=e^0(x), \quad h(0,x)=h^0(x), \quad
  \theta(0,x)=\theta^0(x) \quad \mbox{ in } \Omega\;,
\end{equation}
and boundary conditions
\begin{equation}\label{e:BC}
 \nu\times e =0 \quad \mbox{ and }\quad   \theta=0 \quad \mbox{ in }
 (0,T)\times \partial\Omega\,.
\end{equation}
With $z=(e,h,\theta)^T$, $N$ the matrix 
\begin{equation} \label{e:N}
N= \Mat{ \ve & 0 & a \\ 0& \mu &  0 \\ 
\theta_0 a^\top & 0 & c+\theta_0 \langle \ve^{-1} a,a\rangle}
\end{equation}
and $A$ the unbounded linear operator on $H=L^2(\Omega,\C^7)$ defined by
\begin{equation} \label{e:A}
A = \Mat{ -\sigma & \nabla\times & 0 \\  -\nabla\times &0&0 \\
       0&0&\nabla\cdot \kappa \nabla}
\end{equation}
with domain 
\begin{equation} \label{e:domA}
D(A)= H_0(\mbox{curl},\Omega)\times H(\mbox{curl},\Omega)\times 
\{ \theta \in H^1_0(\Omega)\colon  \nabla\cdot(\kappa\nabla\theta)\in L^2(\Omega)\}\,,
\end{equation}
we rewrite the system \eqref{e:PDE-system} % above 
as 
\begin{equation} \label{e:system_0} 
Nz_t = Az+G\,, \quad \mbox{ where } G=\mat{-j\\0\\r}\,.
\end{equation}

% GOAL

\medskip
Our goal in this work is twofold:

\begin{itemize}

\item[i)]
to introduce a natural functional-analytic setup and ascertain the well-posedness of the initial-boundary value problem (IBVP) \eqref{e:PDE-system}-\eqref{e:IC}-\eqref{e:BC}, a result
which is prerequisite to the stability analysis and yet appears unavailable in the literature;

\item[ii)]
to explore the stability properties of the solutions to the coupled PDE system.

\end{itemize}

% (Sub)section: MAIN results

\subsection{Main results} \label{s:main}
Our first result establishes the well-posedness of the IBVP \eqref{e:PDE-system}-\eqref{e:IC}-\eqref{e:BC} in the function space
\begin{equation} \label{e:basic-space}
H:=L^2(\Omega,\mathbb{C}^7)\,.
\end{equation}
We note that existence and uniqueness of weak solutions was discussed in \cite[Theorems~7.33, 7.34]{FabMor} in the simpler case of constant coefficients, by using the Fourier transform in the time variable.

% THM well-posedness

\begin{theorem}[Well-posedness] \label{t:well-posed}
Let $T>0$.  
For initial data $z^0=(e^0,h^0,\theta^0)\in H$ and $G\in L^1(0,T;H)$ there exists a unique mild solution $z=(e,h,\theta)\in C([0,T],H)$ to the IBVP \eqref{e:PDE-system}-\eqref{e:IC}-\eqref{e:BC}. 
Furthermore, if $z^0\in D(A)$ (see \eqref{e:domA}), $G\in C([0,T],H)$ and either $G\in L^1((0,T),D(A))$ or $G\in W^{1,1}((0,T),H)$, then there exists a unique solution (in a strict sense) $z\in C([0,T],D(A))\cap C^1([0,T],H)$.
\end{theorem} 

\begin{remark}
\begin{rm}
In this article the regularity property $y\in C([0,T],D(A))\cap C^1([0,T],H)$ for the 
solution $y$ to the nonhomogeneous problem $y'=Ay+f$ is subsumed under the adjective {\em strict}, as is more usual within semigroup theory.
We recall that in the said context the concept of {\em classical} solution is weaker, as differentiability in $t_0=0$ is not required; see \cite[Def.~2.1, p.~105]{Pazy} or \cite[Def.~3.1, p.~129]{BDDM}.
We note that the very same membership $z\in C([0,T],D(A))\cap C^1([0,T],H)$ is associated with solutions termed ``classical'' in \cite{ReRo2004}; see the definition on p.~401 and Theorem~12.16 therein.
\end{rm}
\end{remark}

In the case that the conductivity $\sigma$ is uniformly positive definite on $\Omega$, the property of exponential stability can be inferred by using the methods pursued in \cite[Section 2]{Ell19}.
In order to state our result we need to introduce the function space 
\begin{equation*}
H_0(\mathrm{div}_\mu\,0,\Omega) = 
\big\{h\in L^2(\Omega,\C^3)\colon \mathrm{div}(\mu h)=0 \mbox{ in }\Omega
 \mbox{ and } \langle \nu, \mu h \rangle =0 \mbox{ on } \partial \Omega\big\}\,,
\end{equation*}
which is a closed subspace of $H$.
The statement is as follows.

% THM exponential stability

\begin{theorem}[Exponential stability] \label{t:exp-stability}
Suppose that $\Omega$ is simply connected and that $\sigma\in L^\infty(\Omega,\R^{3\times 3}_\mathrm{sym})$ is uniformly positive definite on $\Omega$.
Then there exists a positive constant $\alpha$ such that 
\begin{equation*}
\|(e,h,\theta)\| \lesssim e^{-\alpha t}\|(e^0,h^0,\theta^0)\|
\end{equation*}
for any mild solution to the IBVP \eqref{e:PDE-system}-\eqref{e:IC}-\eqref{e:BC} with 
$h^0 \in H_0(\mathrm{div}_\mu\,0,\Omega)$, $j=0$ and $r=0$. 
Here $\|\cdot\|$ denotes the usual norm in $L^2(\Omega,\C^7)$.
\end{theorem}

We note here that by applying the divergence to the second equation in \eqref{e:PDE-system} and integrating in time, we see that the initial data with $\mathrm{div}(\mu h^0)=0$ produce a magnetic field with $\mathrm{div}(\mu h(t))=0$ for all $t\in [0,\infty)$. This is Gauss's law for the magnetic induction. It is equivalent (in classical physics) to the fact that there are no magnetic charges. 
Likewise, the boundary condition $\langle \nu, \mu h(t) \rangle =0$ will hold as soon as it is enforced at $t=0$. 
\begin{remark}
We like to point out that Gauss's law for the magnetic induction and the boundary condition $\langle \nu, \mu h\rangle=0$ are both necessary for stability of the fields. If $g\in H^1(\Omega)$ has compact support in $\Omega$, then the vector $(0,\nabla g,0) \in H$ solves \eqref{e:PDE-system} for all $t\in [0,\infty)$. 
Likewise $(0,\nabla g,0)\in H$ is a solution to \eqref{e:PDE-system}, provided $g \in H^1(\Omega)$ is a solution to $\mathrm{div}(\mu \nabla g)=0$. This solution satisfies Gauss's law, but it maybe non-trivial without some homogeneous boundary condition.

Finally, if $\Omega$ is not simply connected, then the space $H(\mathrm{curl}\,0,\Omega)\cap H_0(\mathrm{div}_\mu\,0,\Omega)$ is non-trivial and will preclude stability of the semigroup $S(t)$. 
\end{remark}
  
The more interesting question is what happens when $\sigma$ is only semipositive definite 
on $\Omega$. In order to discuss the stability properties of the solutions to the boundary value problem (BVP) \eqref{e:PDE-system}-\eqref{e:BC} in this scenario, we introduce a decomposition of the set $\Omega$ in connection with the definiteness properties of the conductivity parameter $\sigma$. 
Recall that the conductivity is assumed to be semi-positive definite; accordingly we decompose $\Omega$ into disjoint and measurable subsets:
\begin{equation*}
\Omega = \Omega_0 \cup \Omega_{\ge} \cup \Omega_{>}\,,
\end{equation*}
with
\begin{equation*}
\begin{split}
&\Omega_0 := \{x\in \Omega\colon \sigma(x)=0 \;\text{a.e.}\}
\\
&\Omega_>:=\{x\in \Omega\colon \sigma(x)>0 \;\text{a.e.}\}
\\
&\Omega_{\ge}:= \Omega\setminus (\Omega_0\cup \Omega_>)\,.
\end{split}
\end{equation*}

In addition, we will restrict the state space to  
\begin{equation} \label{e:spaceY}
Y= \big\{(e,h,\theta)\in H\colon h \in H_0(\mathrm{div}_\mu \,0,\Omega), \; 
 \mbox{div}(\ve e + a \theta) = 0 \; \text{in $\Omega_0$}\big\}\,,
\end{equation}
which is a closed subspace of $H$. 
From Theorem \ref{t:exp-stability} we are already familiar with the condition on $h$. Here we add also Gauss's law for the electric displacement $\ve e+a\theta$, but only on the set where $\sigma = 0$. In other words, we assume that there are no electric charges present in 
$\Omega_0$. 
Applying the (distributional) divergence to the homogeneous system shows that these two properties are preserved over time, namely, a mild or strict solution with initial data in $Y$ will be reside in $Y$ for all times provided $G\equiv 0$. 

In this case we are led to increase the regularity assumed on the electric permittivity and the magnetic permeability -- from $\ve,\mu \in L^\infty(\Omega,\R^{3\times 3})$ to $\ve,\mu \in W^{1,\infty}(\Omega,\R^{3\times 3})$.
\\
In the sequel, the interior of a subset $C\subset \R^n$ is denoted by $\mathring{C}$.

% THM convergence to equilibria

\begin{theorem}[Convergence to equilibria] \label{t:convergence}
Suppose that $\ve,\mu \in W^{1,\infty}(\Omega,\R^{3\times 3})$ and that $\mathring{\Omega_>}\ne \emptyset$. 
If $z(\cdot)=(e,h,\theta)(\cdot)$ is a mild solution to the IBVP \eqref{e:PDE-system}-\eqref{e:IC}-\eqref{e:BC} with $z^0=(e^0,h^0,\theta^0)\in Y$, $j=0$, and $r=0$, then 
\begin{equation} \label{e:convergence}
z(t) \longrightarrow \tilde{z}\in \ker A\cap Y \quad \textrm{in $H$, as $t\to +\infty$.}
\end{equation}

\end{theorem} 

\smallskip
Further work that expands on the latter result -- including, in particular, a  characterization of $\ker A\cap Y$ -- is found in Section~\ref{s:stability}; see
Proposition~\ref{p:last}. 
This part of our analysis brings about the following results regarding strong stability.

% Strong stability, I

\begin{theorem}[Strong stability, I] \label{t:strong-stability}
Assume -- besides the standing assumptions of Theorem~\ref{t:convergence} -- that $\Omega$ is simply connected,
$\mathring{\Omega_{\ge}}=\emptyset$ and $\mathring{\Omega_>}$ has non-empty interior with a Lipschitz boundary.
If $\partial\Omega_0$ has a sole component, then all mild solutions to problem \eqref{e:PDE-system}-\eqref{e:IC}-\eqref{e:BC} with $z^0\in Y$ tend to $0$, as $t\to+\infty$.
In other words, the system \eqref{e:system_1} is strongly stable.
\end{theorem}

While the last theorem provides a strong stability result based on a topological property
or $\Omega_0$, we put forward a result guaranteeing strong stability based on special initial data.
 
% Strong stability, II 

\begin{theorem}[Strong stability, II] \label{t:strong-stability2}
Assume -- besides the standing assumptions of Theorem~\ref{t:convergence} -- that $\Omega$ is simply connected, $\mathring{\Omega_{\ge}}=\emptyset$ and $\mathring{\Omega_>}$ has non-empty interior with a Lipschitz boundary.
All solutions to problem 
\eqref{e:PDE-system}-\eqref{e:IC}-\eqref{e:BC} with $z^0\in Y$ tend to $0$, as 
$t\to+\infty$, provided $\mathrm{div}(\ve e^0)=0$ in $\Omega$.
In other words, the system \eqref{e:system_1} is strongly stable.
\end{theorem}

% (Sub)section: Derivation of the PDE model

\subsection{Derivation of the PDE model} \label{ss:derivation}
In this section we briefly discuss the constitutive assumptions
underlying the coupled system \eqref{e:PDE-system}. The presentation is
intended only to motivate the model; the subsequent analysis can be
read independently.

We consider a rigid thermoelectromagnetic body occupying a bounded
region $\Omega\subset\mathbb R^3$ with sufficiently smooth boundary
$\partial\Omega$. Since the body is assumed rigid, mechanical effects
are neglected.
The electromagnetic fields satisfy Maxwell's equations
\begin{equation}\label{e:Maxwell-equations}
\begin{aligned}
D_t-\nabla\times h &= -j,
\\
B_t+\nabla\times e &=0,
\end{aligned}
\end{equation}
where $e$ and $h$ denote the electric and magnetic fields,
respectively, while $D$, $B$, and $j$ are the electric displacement,
magnetic induction, and current density.

The local balance of energy is given by
({\sl cfr.}~\cite[pp.~25 and 204]{FabMor})
\begin{equation}\label{energyBalance-5_12}
  w_t
  =
  -\nabla\cdot q
  + r
  + \langle h,B_t\rangle
  + \langle e,D_t\rangle
  + \langle e,j\rangle .
\end{equation}
Here $w$ is the internal energy density, $q$ is the heat flux, and
$r$ is an internal heat source.

We adopt a linear phenomenological constitutive model describing a
thermoelectromagnetic medium near a reference equilibrium temperature
$\theta_0>0$. The constitutive assumptions are consistent with a
quadratic free-energy density of the form
\begin{equation}
  \label{eq:free-energy}
   \psi(e,\vartheta)
  =
  \frac12 \langle \varepsilon e,e\rangle
  +
  \vartheta\, \langle a,e\rangle
  +
  \frac c2 \vartheta^2 ,
\end{equation}
where $\varepsilon$ is the electric permittivity tensor,
$a\in\mathbb R^3$ is a thermal polarization vector,
$c>0$ is an effective heat-capacity coefficient,
and
\[
  \vartheta=\theta-\theta_0
\]
denotes the temperature variation around the reference state. In \eqref{eq:free-energy}, the term $\langle \varepsilon e,e\rangle$ corresponds to the anisotropic electrostatic energy density associated with the permittivity tensor
$\varepsilon$.

The corresponding constitutive relation for the electric displacement
is obtained by differentiating the free-energy density with respect to
the electric field:
\[
  D=\partial_e\psi ,
\]
which yields
\begin{equation}\label{eq:electric-displacement}
  D=\varepsilon e+a\vartheta .
\end{equation}
The coupling term $a\vartheta$ models thermally induced polarization
effects in anisotropic polar media, such as pyroelectric or
ferroelectric materials.

We further assume the constitutive relations
\begin{equation}\label{eq:constitutive-laws}
  B=\mu h,
  \qquad
  j=\sigma e,
\end{equation}
together with Fourier's law
\begin{equation}\label{eq:Fourier-law}
  q=-\kappa\nabla\theta .
\end{equation}

Following \cite[Section~7.19.1]{FabMor}, we adopt the constitutive
relation
\begin{equation}\label{eq:thermal-source}
  s
  =
  \theta_0\langle \varepsilon^{-1}a,D_t\rangle
  +
  c\vartheta_t ,
\end{equation}
where $s$ denotes the thermal source contribution.

Substituting \eqref{eq:electric-displacement} into
\eqref{eq:thermal-source}, we obtain
\[
  s
  =
  \theta_0\langle a,e_t\rangle
  +
  \theta_0\langle \varepsilon^{-1}a,a\rangle \vartheta_t
  +
  c\vartheta_t .
\]
As $\vartheta_t=\theta_t$, this becomes
\begin{equation}\label{eq:thermal-source-expanded}
  s
  =
  \theta_0\langle a,e_t\rangle
  +
  \big(
    c+\theta_0\langle \varepsilon^{-1}a,a\rangle
  \big)\theta_t .
\end{equation}
The positivity assumptions on $\varepsilon$ and $c$ ensure coercivity
of the underlying quadratic free-energy density.
 
Using
\[
  s
  =
  -\nabla \cdot q + r
\]
together with \eqref{eq:Fourier-law}, we obtain
\begin{equation}\label{e:derivation-3rd}
  \big(
    c+\theta_0\langle \varepsilon^{-1}a,a\rangle
  \big)\theta_t
  +
  \theta_0\langle a,e_t\rangle
  -
  \nabla \cdot (\kappa\nabla\theta)
  =
  r .
\end{equation}

On the other hand, substituting
\eqref{eq:constitutive-laws}--\eqref{eq:electric-displacement}
into Maxwell's equations \eqref{e:Maxwell-equations} gives
\begin{equation}\label{e:derivation-1st}
  \varepsilon e_t
  +
  a\theta_t
  -
  \nabla\times h
  +
  \sigma e
  =
  -j
\end{equation}
and
\begin{equation}\label{e:derivation-2nd}
  \mu h_t
  +
  \nabla\times e
  =
  0 .
\end{equation}

Equations
\eqref{e:derivation-1st},
\eqref{e:derivation-2nd},
and
\eqref{e:derivation-3rd}
coincide precisely with \eqref{e:PDE-system}.

The constitutive law \eqref{eq:electric-displacement} should be viewed
as a linearized phenomenological model for anisotropic media exhibiting
thermally induced polarization effects. Examples include pyroelectric
or ferroelectric materials such as lithium niobate and barium
titanate.

\subsection{Previous work}
The question of stabilization of certain exemplary conservative hyperbolic equations such as the wave equation in bounded domains via some kind of damping acting either in the interior of the domain or on its boundary -- possibly, on a portion of it -- is an old problem, with classical and novel ideas explored from the seventies on (with former contributions by Benchimol, Slemrod, Quinn and Russell, Russell, Chen, Lagnese, J.-L. Lions, Haraux, Triggiani, Komornik and Zuazua, and others). 
% As for the case n = 2 we refer to T CHEUGOUE [1].
In the case of a linear wave equation with a damping term effective on the entire domain,
all solutions tend to zero exponentially fast.
If the dissipation is effective only in a measurable subset $\omega$ with positive measure, even in the linear case an exponential decay requires appropriate geometric conditions on
$\omega$; see e.g. \cite{Har_1989} and \cite{Z_1990}.

A great deal of research has been subsequently devoted to the stabilization of wave (and plate) equations with boundary damping; see e.g. the book by Komornik \cite{K1994}.
On this respect, and more specifically as for the geometric constraints, we have to mention the celebrated insightful result by Bardos {\em et al.} \cite{BaLeRa}, relevant to observation, control and stabilization of waves from the boundary.

The next step has been pinpointing the stability properties of wave equations with nonlinear damping, an interesting and challenging problem which is now well understood. 
When strong stability still holds true, a question which naturally arises concerns the type of decay rates of the energy.
\\
(It is outside the scope of this article to provide a thorough overview on the subject of stabilization of wave equations with nonlinear damping.
Seminal work based on the multipliers' method date back to the late 1980s and early 1990s. 
We will mention the paper by Lasiecka and Tataru \cite{LasTat} pertaining to semilinear wave equations with boundary feedback, whose method of proof relies on the construction of an appropriate concave function which controls the growth of the damping at the origin; the final result is very general and describes the decay rates of solutions to semilinear wave equations via the solution to an associated nonlinear ordinary differential equation.
The reader is referred to the review article by Alabau-Boussouira \cite{Alabau2010}, along with the references therein, for a discussion about the subsequent advances -- primarily, yet not exclusively, by the same author -- aimed at providing easily computable and sharper upper energy decay estimates. The line of argument utilized -- and which has proved fruitful over time -- entails the development of appropriate nonlinear integral inequalities pertaining to the energy of the physical system.)

\vspace{1mm}
Coming back to the core topic of this article, it is known that the energy of the (uncoupled) Maxwell system is conserved if the conductivity $\sigma$ vanishes everywhere in $\Omega$;
see e.g.~\cite{Ell19}.
Indeed, the dissipation originates in the current $-\sigma e$ itself, with
this term playing the role of the (sole) damping, whose effectiveness/strenght
depends on the support of $\sigma$.    
An accepted outcome is that exponential stability holds true, provided $\sigma$
is positive definite on the whole $\Omega$; see e.g. \cite[Theorem~2.1]{Ell19}, proved under the assumption that $\Omega$ is simply connected.
Former contributions to the study of this problem are due to Phung \cite{Ph2000} in the special case of constant isotropic coefficients, with their result extended to the variable isotropic coefficients case by Nicaise and Pignotti \cite{nicaise-pignotti2005}.

In this context we want to mention also the recent work by Nutt and Schnaubelt \cite{NS26} where exponential stability of the semigroup generated by Maxwell's equations is established under the hypothesis that $\sigma$ is uniformly positive definite in a collar of the boundary and that the coefficients $\ve$ and $\mu$ satisfy a non-trapping condition.

Pinpointing the property of strong stability for the Maxwell system in the presence 
of {\em localized} conductivity is a much more challenging problem.
For a brief overview of the subject and most recent advances, we refer the reader to the work by Nicaise and Schnaubelt \cite{NS25}, where -- under a special constraint on the support of
$\sigma$ and having set $\ve, \mu=1$ -- it is shown that the semigroup underlying the dynamics is polynomially stable, at rate $1/2$.
The method of proof relies on the renowned, powerful resolvent criterion by Borichev and Tomilov \cite {BT10}.

\vspace{1mm}
The use of certain abstract characterizations % tauberian-type results 
-- say, within the {\em frequency domain} --
such as the Gearhart-Pr\"uss theorem and the aforementioned Borichev-Tomilov criterion 
proves very effective for pinpointing the decay rates of linear PDE systems in the presence
of some form of dissipation.
Indeed, a great deal of favorable results regarding exponential or polynomial stability for significant coupled systems of PDE, quite distinct in nature from one another, have been obtained in the past twenty years or so.
A common feature of these systems is that they comprise a conservative hyperbolic component -- such as an undamped wave equation or system of elasticity -- coupled (possibly, along an interface) with a parabolic-like equation supplying (intrinsic) dissipation; the stability of the overall dynamics indicates that the dissipation has propagated through the coupling.
Typical examples are thermoelastic systems and composite PDE systems which describe acoustic- or fluid-structure interactions (FSI); 
see \cite{AvLa-Trieste} (for a clever use of the multipliers' method, and in the absence of dissipative mechanisms in the Kirchhoff plate equation),   \cite{AvBu2014} (where the exponential decay rates for a 3D-2D flow-structure interaction is established via the Gearhart-Pr\"uss theorem), and \cite{AvLasTrig2016} (attaining the optimal decay rate for a 3D-3D heat-wave model, thus solving a longstanding open problem until then).
It is important to emphasize that applying the said criteria requires that appropriate, model-specific PDE estimates (for a corresponding static PDE system) are established; to this end, microlocal analysis might be necessitated and eventually come into play, like e.g. in \cite{AvLasTrig2016}.

Our PDE system \eqref{e:PDE-system} is linear and couples the hyperbolic Maxwell equations with the parabolic heat equation over the same region. This type of coupling has been studied in the past, in particular the coupling of the (isotropic) elastic wave equations with the heat equation. Starting with the works by Dafermos \cite{Da68}, criteria for strong and exponential stability were studied in \cite{HLP93}, \cite{LebZua}, and \cite{Ko00}. 
At issue is here that the dissipation impacts only the pressure waves. 
Likewise, in our system the dissipation acts only the electric field and not the magnetic field.

% OUTLINE

\subsection{An outline of the paper}
The paper is organized as follows.
In Section~\ref{s:functional-analysis} we introduce notations and set up a functional-analytic framework for the IBVP \eqref{e:PDE-system}-\eqref{e:IC}-\eqref{e:BC}.
The grounds and tools for the well-posedness result stated in Theorem~\ref{t:well-posed} are
provided by semigroup theory; see Proposition~\ref{p:Lumer-P}.    
  
The analysis of the issue of stability is carried out in Section~\ref{s:stability}.
Primarily based upon PDE methods, it brings about the respective proofs of the Theorems~\ref{t:exp-stability}, \ref{t:convergence}, \ref{t:strong-stability}, \ref{t:strong-stability2}.

Starting from energy identities pertaining to solutions in a mild or strict sense, an observability inequality is key to the proof of exponential stability,
under the assumption that the conductivity $\sigma$ is positive definite on $\Omega$.

When $\sigma$ is only semipositive definite on $\Omega$, higher regularity of the electric permittivity $\varepsilon$ and the magnetic permeability $\mu$ is called for. 
Then, a preparatory key step towards convergence to equilibria %Theorem~\ref{e:convergence}
is provided by Proposition~\ref{p:pre-strong}. 
The additional work in Section~\ref{ss:expand} fosters a deeper understanding of the set of equilibria, thereby achieving strong stability in specific either topological or analytic setting.

Lastly, additional remarks on the constitutive assumptions, admissible
materials, and related literature are collected in an Appendix.

% Section: PRELIMINARIES, WELL-POSEDNESS

\section{Functional-analytic setup. Well-posedness} \label{s:functional-analysis}
In this section, we introduce notation and set a functional-analytic framework where
it is shown that the IBVP is well-posed.

\smallskip
\noindent
{\em Symmetrization of the system.} 
Recall the original system in \eqref{e:system_0}, i.e. $Nz_t=Az+G$, with $N$ and $A$ as
in \eqref{e:N} and \eqref{e:A}, respectively. % and $G$ itself defined in \eqref{e:system_0}.
Taking  
\begin{equation} \label{e:diagonal-U}
U=\Mat{I_6 &0\\0&\theta_0^{1/2}}\,.
\end{equation}
we see that 
\begin{equation} \label{e:matrixM}
U^{-1}NU =\Mat{ \ve & 0 & \sqrt{\theta_0}a \\ 0& \mu &  0 \\ 
\sqrt{\theta_0} a^\top &0 & c+\theta_0 \langle \ve^{-1} a,a\rangle}=:M\,, 
\end{equation}
and for the new dependent variable $y= U^{-1}z$ it holds
\begin{equation*}
My_t = Ay+F\,, \quad \text{with $F=U^{-1}G$.}
\end{equation*}
Using the Schur complement formula (see e.g. \cite[Prop.~3.9]{Ser10}), we compute 
\begin{equation*}
\det M= c(\det \ve)(\det \mu)
\end{equation*}
and observe also that the eigenvalues of $M$ are the eigenvalues of $\ve$, $\mu$ as well as $c$. 
Hence, the matrix $M$ is positive definite and we can introduce on the state space $H$
the inner product
\begin{equation}\label{e:inner}
\langle z,w \rangle_{H} = \int_\Omega \langle Mz,w\rangle\,dx\,.
\end{equation}
The norm $\|\cdot\|_H$ induced by this inner product is equivalent to the one induced by the Euclidean inner product.
Although the explicit form of the inverse $M^{-1}$ may not be needed in the computations that follow, we provide it below for completeness and the readers' convenience: % but
\[
M^{-1} = \Mat{\ve^{-1} + \theta_0 \ve^{-1}aa^\top \ve^{-1}/c
&0& -\sqrt{\theta_0}\ve^{-1} a/c \\ 0&\mu^{-1} &0 \\ 
-\sqrt{\theta_0}a^\top \ve^{-1}/c    &0& c^{-1}}\,.
\] 
Still, we write the system in the form
\begin{equation} \label{e:system_1}
y_t = M^{-1}Ay + M^{-1}F\,.
\end{equation}

\smallskip
As the IBVP \eqref{e:PDE-system}-\eqref{e:IC}-\eqref{e:BC} is equivalent to a Cauchy problem for the abstract system \eqref{e:system_1}, a result of well-posedness will follow from a generation result for the operator $M^{-1}A$ in $H$. 
This is accomplished via the Lumer-Phillips Theorem. 

% LUMER-PHILLIPS

\begin{proposition} \label{p:Lumer-P}
Let $M$ and $A$ be the linear operators defined in \eqref{e:matrixM} and \eqref{e:A}-\eqref{e:domA}, respectively.
Then the operator $M^{-1}A$ is densely defined, closed, and $m$-dissipative on $H$, namely, 
\begin{itemize}
\item
$\Re \langle M^{-1}Ay,y \rangle_H \le 0$ for all $y\in D(A)$, and 
\item
for all $F \in H$ the equation $(I-M^{-1}A)y=F$ has a unique solution in $D(A)$. 
\end{itemize}
\end{proposition}
 
\begin{proof}
Since $C^\infty_0(\Omega,\C^7)$ is dense in $H$ and $C_0^\infty(\Omega,\C^7)\subset D(A)$, the operator $A$ is densely defined. 
Furthermore, if $y_n\to y$ and $M^{-1}Ay_n\to z$ in $H$ for some sequence $\{y_n\}\subset D(A)$, then for all $\varphi \in C_0^\infty(\Omega, \C^7)$, 
\begin{equation*}
\langle z,\varphi \rangle_H= 
\lim_{n\to \infty}\langle M^{-1}Ay_n,\varphi \rangle_H =
\lim_{n\to \infty}\langle y_n,M^{-1}A^\top \varphi\rangle_H =
\langle y,M^{-1}A^\top \varphi \rangle_H
\end{equation*}
where 
\begin{equation}\label{e:adjoint-of-A}
A^\top = \Mat{ -\sigma & -\nabla\times & 0 \\  \nabla\times &0&0 \\
       0&0&\nabla\cdot \kappa \nabla}\;.
\end{equation}

\noindent
\smallskip
Hence, $Ay=z$ and thus $y\in D(A)$ and $A$ is closed. 

Regarding the dissipativity, one computes 
\begin{equation*}
\Re \langle M^{-1}Ay,y\rangle_H 
% added
= \Re\int_\Omega \langle A y,y\rangle\,dx
=-\int_\Omega \langle \sigma y_1,y_1\rangle\,dx 
- \int_\Omega \langle \kappa \nabla y_3,\nabla y_3\rangle\,dx\,,
 \end{equation*}
which establishes that $A$ is dissipative in $H$. 
Next we solve the equation $(M-A)y = f$ for $f\in H$. 
For $y_1,z_1\in H_0(\mathrm{curl},\Omega)$ and $y_3,z_3\in H^1_0(\Omega)$ consider the bilinear form
\[
 \begin{split}
 B(y,z) =& \int_\Omega \langle \ve y_1,z_1\rangle \,dx
 + \sqrt{\theta_0}\int_\Omega y_3 \langle a,z_1\rangle\,dx +
 \int_\Omega \mu^{-1} \langle \nabla\times y_1,\nabla\times z_1\rangle\,dx 
 \\
 & + \int_\Omega \langle \sigma e,e\rangle\,dx
 + \int_\Omega (c+\theta_0 \langle \ve^{-1} a,a\rangle) y_3\overline{z_3}\,dx 
 + \sqrt{\theta_0} \int_\Omega \langle y_1,a\rangle \overline{z_3}\,dx 
 \\
 & + \int \langle \kappa \nabla y_3,\nabla z_3 \rangle\,dx\,;
 \end{split}
\]
this bilinear form is continuous and coercive on 
$H_0(\mathrm{curl},\Omega)\times H^1_0(\Omega)$, since 
\[
 \begin{split}
 B(y,y) =& \int_\Omega \langle \ve^{1/2}y_1+\sqrt{\theta_0}y_3 \ve^{-1/2}a,  \ve^{1/2}y_1+\sqrt{\theta_0}y_3 \ve^{-1/2}a\rangle\,dx + c \int_\Omega |y_3|^3\,dx 
 \\
 & + \int_\Omega \langle \sigma e,e\rangle\,dx
 + \int_\Omega \mu^{-1} \langle \nabla\times y_1,\nabla\times y_1\rangle\,dx
 + \int_\Omega \langle \kappa \nabla y_3,\nabla z_3 \rangle\,dx\,.
 \end{split}
\] 
By the triangle inequality  
\[
 \alpha \langle \ve^{1/2}y_1+\sqrt{\theta_0}y_3 \ve^{-1/2}a,  \ve^{1/2}y_1+\sqrt{\theta_0}y_3 \ve^{-1/2}a\rangle \ge 
 \frac{\alpha}{2} \langle \ve y_1,y_1 \rangle 
 - \alpha \theta_0 |y_3|^2 \langle \ve^{-1}a,a\rangle 
\]  
and hence, there exists an $\alpha \in (0,1]$ such that 
\[
 c -\alpha \theta_0 \langle \ve^{-1}a,a\rangle \ge 0
\]
for all $x\in \Omega$. 
This way one obtains
\[
 B(y,y) \gtrsim  \|y_1\|^2_{H_0(\mbox{curl},\Omega)}+ \|y_3\|^2_{H^1_0(\Omega)}\,.
\]   
Hence, for any $f=(f_1,f_2,f_3)\in H$ there exists a unique $(y_1,y_3)\in H_0(\mathrm{curl},\Omega)\times H^1_0(\Omega)$ such that 
\[
 B(y,z) = \int_\Omega \langle f_1,z_1\rangle + \langle \mu^{-1}f_2,\nabla \times z_1\rangle + f_3\overline{z_3}\,dx\,.
\]  
Define now $y_2 \in L^2(\Omega,\C^3)$ by 
\[
 y_2 =- \mu^{-1}\nabla \times y_1 + \mu^{-1} f_2\,.
\]
Choosing $z_1 \in C_0^\infty(\Omega, \C^3)$ and $z_3=0$ we infer that 
\[
 \ve y_1 + \sqrt{\theta_0} a y_3-f_1  
 =\nabla\times (-\mu^{-1}\nabla \times y_1+\nabla\times (\mu^{-1}f_2)
\]
in the sense of distributions. 
Hence $y_2 \in H(\mbox{curl},\Omega)$ and 
\[
 \ve y_1 + \sqrt{\theta_0} ay_3 +\nabla\times y_2 = f_1\,.
\]    
Choosing $z_1=0$ and $z_3\in C_0^\infty(\Omega,\C)$ we infer that
\[
 (c+\theta_0 \langle \ve^{-1} a,a\rangle) y_3 + 
 \sqrt{\theta_0} \langle y_1,a\rangle - \nabla\cdot (\kappa \nabla y_3) =f_3
\]
in the sense of distributions. 
We have proved $(y_1,y_2,y_3)\in D(A)$ and that $(M-A)y=f$, which establishes the second assertion, thus concluding the proof.   
\end{proof}

{\em Proof of Theorem~\ref{t:well-posed}.}
By Proposition~\ref{p:Lumer-P} and the Lumer-Phillips Theorem the operator $M^{-1}A$ generates a $C_0$-semigroup $\{S(t)\}_{t\ge 0}$ of contractions in $H$. 
Hence, for $y^0\in H$, $G\in L^1((0,T),H)$ there exists a unique mild solution $y\in C([0,T],H)$ to the IBVP \eqref{e:PDE-system}-\eqref{e:IC}-\eqref{e:BC}.
Furthermore, if $y^0\in D(A)$, $G \in C([0,T],H)$ and either $G\in L^1((0,T),D(A))$ or $G\in W^{1,1}((0,T),H)$, then there exists a unique strict -- and hence classical -- solution $y\in C([0,T],D(A))\cap C^1([0,T],H)$. Here $T>0$.

Since $z=Uy$, with $U$ the diagonal matrix defined by \eqref{e:diagonal-U}, the conclusion follows.
\qed

% Section: LONG TERM BEHAVIOUR

\section{Long-time behavior} \label{s:stability}
From now on we work with the homogeneous system, that is $j=0$ and $r=0$ in \eqref{e:PDE-system}. 
The conditions on the coefficient functions are the same as stipulated in the introduction.

\smallskip
\noindent
{\em Energy identities.} 
We will work with the variable $y=(y_1,y_2,y_3) = (e,h,\theta_0^{1/2}\theta)$.
Indeed, owing to the simple relation between $y$ and $z$, the stability properties
established in regard to $y$ are plainly inherited by the actual state variable $z$.
By using a standard argument, we know that strict solutions satisfy the energy identity
\begin{equation}\label{e:energy-id}
 \frac{1}{2}\|y(t_2)\|_H^2 - \frac{1}{2}\|y(t_1)\|_H^2 =
 -\int_{t_1}^{t_2} \int_\Omega \langle \sigma y_1,y_1\rangle\,dx\,dt 
  - \int_{t_1}^{t_2} \int_\Omega \langle \kappa \nabla y_3,\nabla y_3\rangle\,dx\,dt\,,
\end{equation}
for all $0\le t_1<t_2<\infty$. 
Since each mild solution can be obtained as a limit of strict solutions, the identity
\eqref{e:energy-id} implies that mild solutions satisfy 
\begin{equation*}
\langle \kappa \nabla y_3,\nabla y_3\rangle \in L^1((t_1,t_2)\times \Omega)
\end{equation*}
and consequently, because of the uniform positivity of $\kappa$,
\begin{equation*}
\nabla y_3 \in L^2((0,\infty)\times \Omega)=L^2((0,\infty),L^2(\Omega))
\end{equation*} 
and thus 
$y_3 \in L^2((0,\infty),H^1_0(\Omega))$. 
The identity \eqref{e:energy-id} is valid for mild/weak solutions.
Note that it also implies that 
\begin{equation*}
\begin{cases}
y_3(t) \longrightarrow 0 
\\
\sigma^{1/2} y_1(t) \longrightarrow 0
\end{cases}
 \; 
\mbox{ a.e. in } H^1_0(\Omega) \mbox{ and } L^2(\Omega,\C^3), 
\mbox{ respectively, as }t\to +\infty\,.
\end{equation*}  
Since $y_1 \in C([0,\infty),L^2(\Omega,\C^3))$, the convergence $\|\sigma^{1/2} y_1(t)\|_{L^2(\Omega)} \longrightarrow 0$ is pointwise. 

\begin{remark}
 Our notion of energy here is different from the physical energy discussed in Section \ref{ss:derivation}. Here we take as energy the weighted $L^2$-norm of $y$. In contrast, formula (\ref{energyBalance-5_12}) stipulates as energy the integral of $w_t$ over spacetime. 
\end{remark}  
Similarly, since time derivatives of strict solutions are mild/weak solution, we have the energy identity
\begin{equation}\label{e:identity2}
\begin{split}
& \frac{1}{2}\|\partial_ty(t_2)\|_H^2 
 - \frac{1}{2}\|\partial_ty(t_1)\|_H^2 =
\\[1mm] 
& \myspace 
-\int_{t_1}^{t_2} \int_\Omega 
 \langle \sigma \partial_ty_1,\partial_ty_1\rangle\,dxdt 
  - \int_{t_1}^{t_2} \int_\Omega 
  \langle \kappa \nabla \partial_ty_3,\nabla \partial_ty_3\rangle\,dxdt\,,
\end{split}
\end{equation}
for strict solutions with the additional regularity $y_3 \in H^1((0,\infty),H^1_0(\Omega))$.
Since strict solutions satisfy $y_3 \in C([0,\infty),H^1_0(\Omega))$, we have pointwise convergence 
\begin{equation*}
\|y_3(t)\|_{H^1(\Omega)}\longrightarrow 0\,.
\end{equation*} 
Furthermore, $\|\nabla \partial_ty_3\|^2_{L^2(\Omega)} \longrightarrow 0$ a.e., as $t\to +\infty$.

As anticipated in Section~\ref{s:main}, we restrict %ourseleves 
to the function space $Y$ defined in \eqref{e:spaceY}. 
We observed already then that the $C_0$-semigroup $S(t)$ maps $Y$ into itself. The generator of $S(t)\big|_Y$ is then the densely defined and closed operator 
\[
 B := M^{-1}A\big|_{Y}\colon Y\to Y \quad \mbox{ with } 
 \quad D(B)=D(A)\cap Y\;.
\] 
Abusing notation, the restriction of the semigroup $S(t)$ to $Y$ will be denoted also by 
$S(t)$.

% HERE THM on EXP STABILITY
For the proof we need to work with a Helmholtz decomposition of the electromagnetic field
$(y_1,y_2)$. 

\begin{proposition}\cite[Prop.~2]{Ell19} \label{p:helmholtz}
Suppose that $\Omega$ is simply connected and that $y=(y_1,y_2,y_3) = (e,h,\theta_0^{1/2}\theta)\in C([0,\infty),Y)$ is a mild solution to system \eqref{e:system_1}.
Then there exist vector fields 
\begin{equation*}
a\in C([0,\infty),H^1_{{\rm t}0}(\Omega)) \cap C^1([0,\infty),L^2(\Omega,\mathbb{C}^3))\,,
\; 
h_2 \in C([0,\infty),\mathbb{H}_2(\Omega))\,,
\end{equation*}
and a scalar function $p\in C([0,\infty),H_0^1(\Omega))$ such that
\begin{equation} \label{e:orthogonal}
\mu y_2 = \nabla \times a \quad \mbox{ and } \quad y_1= -\partial_t a + \nabla p+h_2\,.
\end{equation} 
Furthermore, $\|a(t)\|_{L^2(\Omega)} \lesssim \|\nabla\times a(t)\|_{L^2(\Omega)}$, uniformly on compact time intervals.

Here $H^1_{{\rm t}0}(\Omega)=H_0(\mathrm{curl},\Omega)\cap H(\mathrm{div},\Omega)$ and
$\mathbb{H}_2 = H_0(\mathrm{curl}\,0,\Omega)\cap H(\mathrm{div}\,0,\Omega)$.
\end{proposition}
Now we turn to the proof of Theorem~\ref{t:exp-stability}.

\medskip
\noindent
{\em Proof of Theorem~\ref{t:exp-stability}.} 
We will establish the observability inequality
\begin{equation} \label{e:obs-ineq}
 \int_{\delta}^{T-\delta} \int_\Omega \langle \mu y_2,y_2\rangle \,dx\,dt
 \lesssim 
 \int_0^T \int_\Omega \langle \sigma y_1, y_1\rangle \,dx\,dt +
 \int_0^T \int_\Omega |y_3|^2\,dx\,dt\,.
\end{equation}
Since any mild solution is also a weak solution, we know that
\[
\int_0^T \Big[\int_\Omega \big(\langle\psi_t,\ve y_1+ay_3\rangle + \langle \nabla\times \psi,y_2 \rangle  - \langle \psi,\sigma y_1 \rangle\big) \,dx \Big]\,dt= 0
\]
for all $\psi \in C([0,T],H^1_{{\rm t}0}(\Omega))\cap C^1([0,T],L^2(\Omega,\mathbb{C}^3))$ satisfying $\psi(0)=\psi(T)=0$. Choosing $\varphi \in C_0^\infty(0,T)$ with $\varphi \equiv 1$ on $(\delta,T-\delta)$ and $\psi =\varphi^2 a$, where $a$ is the function introduced in the
Proposition~\ref{p:helmholtz}, gives for any $\alpha>0$
\[ 
\begin{split}
  \int_\delta^{T-\delta} \int_\Omega \langle \mu y_2,y_2 \rangle \,dx\,dt
  \le & \int_0^T \int_\Omega \varphi^2(t)\langle \nabla\times a,y_2 \rangle \,dx\,dt
  \\
  =&-\int_0^T \int_\Omega \varphi^2(t) \langle a_t, \ve y_1+ay_3\rangle \,dx\,dt
  \\
 & -\int_0^T \int_\Omega 2\varphi'(t)\varphi(t) \langle a, \ve y_1+ay_3 \rangle \,dx\,dt
   \\
 &+ \int_0^T \int_\Omega \varphi^2(t) \langle a,\sigma y_1\rangle \,dxdt \\
  \le& \,C_\alpha \int_0^T  \|y_1(t)\|_{L^2(\Omega)}^2+\|y_3(t)\|_{L^2(\Omega)}^2\,dt 
  +\int_0^T \|a_t(t)\|_{L^2(\Omega)}^2\,dt 
  \\& + \alpha \int_0^T\int_\Omega \varphi^2(t)|a|^2 \,dx\,dt
\end{split}
\]
where the Cauchy-Schwarz inequality was used in the last step. 
The last two integrals can estimated using Proposition~\ref{p:helmholtz}. For $\alpha$ sufficiently small we arrive at
\[
\begin{split}
\int_0^{T} \int_\Omega \varphi^2(t) \langle \mu y_2,y_2 \rangle \,dx\,dt
&=\int_0^T \int_\Omega \varphi^2(t)\langle \nabla\times a,\mu^{-1} (\nabla\times a) \rangle \,dx\,dt
\\[1mm]
&\lesssim  \int_0^T \int_\Omega \big(\langle y_1,y_1 \rangle+|y_3|^2\big)\,dx\,dt\,.
\end{split}
\]
Thus, the integral on the right-hand side can be estimated by some constant times
\[
\int_0^T \int_\Omega \big(\langle \sigma y_1,y_1 \rangle+|y_3|^2\big)\,dx\,dt
\]
since $\sigma$ is uniformly positive definite.
The observability inequality \eqref{e:obs-ineq} follows, as a consequence.
  
Finally, \eqref{e:obs-ineq} can be combined with the energy identity 
\eqref{e:energy-id} to obtain exponential stability. 
\qed

\medskip

% Strong stability

We now move on to the interesting case when $\sigma$ is only semipositive definite on 
$\Omega$. 
A key point in the proof of Theorem~\ref{t:convergence} is the following result. 

\begin{proposition} \label{p:pre-strong}
Suppose that $\ve,\mu \in W^{1,\infty}(\Omega,\R^{3\times 3})$ and that $\mathring{\Omega_>}$ is non-empty. 
If $y(\cdot)$ is a strict solution to system \eqref{e:system_1} with $y^0\in Y$, then 
\begin{equation*}
\|y_t(t)\|_H \longrightarrow 0\,, \quad \textrm{as $t \to +\infty$.}
\end{equation*}

\end{proposition}

\begin{proof}
By density (see, e.g. Chapter II, Proposition 1.8 \cite{EngNag}) it will suffice to work with solutions corresponding to $y^0 \in D(B^2)\cap Y$, where 
\[
 D(B^2) = \{y\in D(B)\colon By\in D(B)\}\,.
\] 
By the energy identity \eqref{e:identity2}, the function $\|y_t(t)\|_H$ is decreasing. We set $y_t=:w$ to simplify the notation, and note that $w \in C([0,\infty),D(B))\cap C^1([0,\infty),Y)$. 
Arguing by contradiction, suppose that $\|w(t)\|_H\longrightarrow \beta>0$, as $t\to +\infty$. Since $w$ is a strict solution to the BVP \eqref{e:PDE-system}-\eqref{e:BC}, by the energy identity \eqref{e:identity2} we see that $\|w_t(t)\|_H$ remains bounded for all $t>0$. 
Using now the differential equations \eqref{e:PDE-system} we infer that the $L^2(\Omega)$-norms of $\nabla\times w_1$, $\nabla\times w_2$ and $\nabla\cdot (\kappa \nabla w_3)$ are bounded as well, for all $t>0$.  

So far we have shown that $w_2(t)$ is in a bounded set of $H(\mathrm{curl},\Omega)\cap H_0(\mathrm{div}_\mu,\Omega)$ and $w_3(t)$ is in a bounded set of $H^1_0(\Omega)$ for all $t>0$. Both of these spaces are compactly imbedded in $L^2(\Omega,\C^3)$ and $L^2(\Omega,\C)$, respectively. The former is a consequence of \cite[Theorem~8.1.1]{AssCiaLab} where the latter is the Rellich Selection Theorem. 

Regarding $w_1$, we know that $w_1(t)$ remains bounded in $H_0(\mathrm{curl},\Omega)$, so we have to discuss its divergence. Applying the divergence to the first equation in 
\eqref{e:PDE-system} gives
\[
\mathrm{div}(\ve w_1(t)) = -\mathrm{div}\big[a \theta_0^{1/2} w_3(t)+\sigma y_1(t)\big]\,.
\]  
Consider now the elliptic equation 
\[
 \mathrm{div}(\ve \nabla g) = \mathrm{div}(\ve w_1)
\]
complemented with zero Dirichlet conditions in $\Omega_0$. 
This problem has a unique solution $g=g(t) \in H^1_0(\Omega)$ and setting $f=z_1 -\nabla g$ we obtain a Helmholtz decomposition 
\[
w_1(t) = f(t) + \nabla g(t)\,,
\]
with $f(t) \in H(\mbox{div}_\ve\,0,\Omega)$. 
From $w_1(t) \in H_0(\mathrm{curl},\Omega)$ we infer that $\nu\times f(t)=0$ on $\partial\Omega$ for all $t>0$. 

Relying on our discussions after the energy identities \eqref{e:energy-id} and \eqref{e:identity2}, we have 
\begin{equation*}
\begin{split}
 \|\mathrm{div}(\ve w_1(t))\|_{H^{-1}(\Omega)} &=
 \|\mathrm{div}[a \theta_0^{1/2} w_3(t)  +\sigma y_1(t)]\|_{H^{-1}(\Omega)} 
 \\
 &\le \|a \theta_0^{1/2} w_3(t)+\sigma y_1(t))\|_{L^2(\Omega)}
 \longrightarrow 0\,,
\end{split}
\end{equation*}
as $t\to +\infty$, where we used also that $\sigma\lesssim \sigma^{1/2}$ is the sense of positive semidefinite matrices.
Thus
\[
 \|g(t)\|_{H^1(\Omega)} \longrightarrow 0, \quad \mbox{ as } t\to +\infty\,.
\] 
In summary, there exists a sequence $t_n \to +\infty$ such that $w(t_n) \to w\in Y$. 
In addition, $\mathrm{div}(\ve w_1)=0$ in $\Omega$.

Finally, we solve the IBVP \eqref{e:PDE-system}-\eqref{e:IC}-\eqref{e:BC} with initial data $w$. The solution of this IBVP, denoted by $W(t)$, satisfies the relation
\[
 W(t) = S(t)w = S(t)\lim_{n\to +\infty} S(t_n)w(0) = 
 \lim_{n\to \infty} S(t+t_n)w(0).
\]  
Hence, by the energy identity $\|W(t)\|_H=\beta$ for all $t$ and thus $W_3(t)=\sigma W_1(t)=0$ for all $t$. We infer that $(W_1(t),W_2(t))$ is a mild solution to Maxwell's equations with $\sigma=0$ which also satisfies the divergence equations $\mathrm{div}(\ve W_1(t))=\mathrm{div}(\mu W_2(t))=0$ in $\Omega$.\footnote{This argument is known as the LaSalle Invariance Principle, see e.g. \cite[Prop. 1.3.6]{Alabau2010}.} 

For $W_1(t),W_2(t)$ we use a unique continuation argument very similar to the one put forward in 
\cite[p.~355-356]{Ell19} to show that $w\equiv 0$.

For that we use the Maxwell operator $M(x,D)u = (\ve^{-1}\nabla\times u_2,-\mu^{-1}\nabla\times u_1)$ on the state space \[
\cH:=H(\mathrm{div}_\ve 0,\Omega)\times H_0(\mathrm{div}_\mu 0,\Omega),
\]
equipped with $L^2(\Omega,\C^6)$ inner product weighted by the positive definite real-symmetric matrix function $\mathrm{diag}(\ve,\mu)$,
with domain 
\[
 D(M)= [H_0(\mathrm{curl},\Omega)\times H(\mathrm{curl},\Omega)]\cap \cH.
\]
This operator is skew-adjoint. The arguments put forward in the proof of Proposition \ref{p:Lumer-P} show that $M^{-1}$ maps $\cH$ continuously into $D(M)$, and since $D(M)$, considered as a Hilbert space with the graph norm, is compactly imbedded in $\cH$ \cite[Theorems~8.1.1, 8.1.3]{AssCiaLab}, the operator $M^{-1}$ is compact. Hence, $M$ has only purely imaginary (non-zero) eigenvalues and the corresponding eigenfunctions are mutually orthogonal in $\cH$. Using these eigenvalues and eigenfunctions, the solution $u(t)=(W_1(t),W_2(t))$ has the expansion
\[
 u(t,x) = \sum_{k=1}^\infty e^{\lambda_k t}   u_k(x)\quad 
 \mbox{ with } \langle u_k,u_l \rangle_{\cH}=0 \mbox{ and }
 \lambda_k\neq \lambda_l.
\]  
We introduce the sequence
\[
  v_k(t,x) = \frac{1}{t} \int_0^t u(s,x) e^{-\lambda_k s}\,ds.
\] 
and observe that the first three component functions of $w_k$ must vanish on $\Omega_>$. Using the series expansion above gives
\[
 v_k(t,x) = u_k(x) + \frac{1}{t}\sum_{l\neq k} 
 \frac{e^{(\lambda_l-\lambda_k)t}-1}{\lambda_l-\lambda_k}u_l(x).
\] 
Letting $t\to +\infty$ shows that the first three component functions of $u_k$ must vanish in $\Omega_>$. Each eigenfunction satisfies $\lambda_k u_k = Mu_k$ and using the structure of the Maxwell operator, we see that the remaining three components of $u_k$ must vanish on $\Omega_>$ for $k=1,2,\dots$ as well. Thus $u_k\equiv 0$ on $\Omega_>$.  
     
Each eigenfunction $u_k$ is a solution to the time-harmonic Maxwell system. Since $u_k=0$ on $\mathring{\Omega_>}$, unique continuation of the time-harmonic Maxwell equations (with zero divergence conditions) implies that $u_k\equiv 0$ in $\Omega$ \cite{Vog91,EY06}, for $k=1,2,...$. Here the Lipschitz continuity of the coefficients is critical. We infer $u(t,x)\equiv 0$. The conclusion $\|y_t(t)\|_H \longrightarrow 0$, as $t\to +\infty$, follows.
\end{proof}

% COROLLARY

\begin{corollary}\label{c: strong}
Under the hypothesis of Proposition \ref{p:pre-strong}, given 
$y^0 \in R(B)=\{Bz\colon z\in D(B)\}$, we have $\|S(t)y^0\|_H \longrightarrow 0$, 
as $t\to +\infty$.
\end{corollary}

\begin{proof}
By assumption $y^0 = Bz^0$ for some $z^0\in D(B)$. Then, 
\[
 S(t)y^0 = S(t)Bz^0 = BS(t)z^0 =\frac{d}{dt}\big[ S(t) z^0\big]
\]
and $\|S(t)y^0\|_H \longrightarrow 0$, as $t\to +\infty$ by Proposition~\ref{p:pre-strong}.
\end{proof}

With this Corollary in hand, we can now prove Theorem~\ref{t:convergence}.

\medskip
\noindent
{\em Proof of Theorem~\ref{t:convergence}.} 
One can show that the adjoint operator $B^\ast$ has domain $D(B)$ and that $B^\ast=M^{-1}A^\top$, see \eqref{e:adjoint-of-A}. 
This leads to $\ker B=\ker B^*$ and thus
 \[
  R(B)^\perp = R(B^\ast)^\perp = \ker B\;.
 \]   
Given $y^0 \in Y$ we can write $y^0 = \xi+ z^0$ where $\xi \in \ker B$ and $z^0 \in \overline{R(B)}$. Since $S(t)$ restricted to $\ker B$ is just the identity, we have  
\[
 S(t) y^0 = \xi + S(t)z^0 .
\]
Given $\ve>0$, there exists a $z\in R(B)$ such that $\|z-z^0\|_H<\ve/2$. Hence, for large $t$, using the Corollary above, results in
\[
 \|S(t) y^0-\xi\|_H\le \|S(t)(z^0-z)\|_H+\|S(t)z\|_H<\ve\,.
\] 
\qed

% EXPANDING on the CONVERGENCE to EQUILIBRIA RESULT

\subsection{Characterization of $\boldsymbol{\ker B=\ker A \cap Y}$}  \label{ss:expand}
%Next we will 
In this section we expand a bit on the earlier result on convergence to equilibria,
by showing that under some additional assumptions the kernel of $A$ is finite-dimensional. 
The next couple of results are inspired by \cite[Lemmas 3.2 and 3.3]{NS25}.
We start with the following one. 

\begin{proposition} 
Suppose that $\mathring{\Omega_>}$ is non-empty. Then the operator 
\begin{equation*}
\mathrm{i} \lambda  - B\colon D(B) \longrightarrow Y
\end{equation*}
is injective for each $\lambda \in \R\setminus \{0\}$. 
\end{proposition}

\begin{proof}
Let $\lambda \in \R\setminus \{0\}$ be given, and take $y \in \ker (\mathrm{i} \lambda - B)$. Then 
\begin{equation*}
\begin{split}
 0 &= \Re \langle (\mathrm{i}\lambda -B)y,y\rangle_H 
 =  \Re\langle (\mathrm{i} \lambda M - A)y,y\rangle_{L^2(\Omega)}
 \\
 &= -\int_\Omega \langle \sigma y_1,y_1\rangle\,dx 
  - \int_\Omega \langle \kappa \nabla y_3,\nabla y_3\rangle\,dx\,.
\end{split}
\end{equation*}
Hence, $y_3= 0$, $\sigma y_1= 0$ on $\Omega$, and $y_1= 0$ on $\Omega_>$. Furthermore, since $\mathrm{i}\lambda \ve y_1 -\nabla\times y_2= 0$ in $\Omega$ we see that $\nabla\cdot (\ve y_1)=0$ in $\Omega$. Likewise, the equation $\mathrm{i}\lambda\mu y_2 + \nabla\times y_1=0$ gives $\mathrm{div}(\mu y_2)=0$ and also $y_2=0$ on $\Omega_>$. The proof is finished using unique continuation for the stationary Maxwell equations (see \cite{Vog91,EY06}). 
\end{proof}  

In the case $\lambda=0$ we need additional hypotheses, thereby achieving 
a precise description of $\ker B$.
 
\begin{proposition} \label{p:last}
Suppose that $\Omega$ is simply connected, $\mathring{\Omega_{\ge}}=\emptyset$ and that $\mathring{\Omega_>}$ has a Lipschitz boundary. 
Then 
\begin{equation} \label{e:kernelB}
\ker B = \big\{(y_1,0,0)\colon y_1|_{\Omega_>}=0 \mbox{ and } y_1|_{\Omega_0} \in H_0(\mathrm{curl}\,0,\Omega_0)\cap H(\mathrm{div}_\ve 0,\Omega_0)\big\}\,.
\end{equation}
The space $\ker B$ is finite dimensional with dimension equal to the number of components for $\partial \Omega_0$.  
 
\end{proposition}

\begin{proof}
Let $y\in \ker B$. As in the proof above we have $y_3=0$ and $\sigma y_1=0$ in $\Omega$. Hence, $\nabla\times y_2=0$ and thus $y_2=0$ in $\Omega$ (since $\Omega$ is simply connected). 
Also $\nabla\times y_1=0$ in $\Omega$ and $y_1=0$ on $\Omega_>$. Then $\mathrm{div}(\ve y_1)=0$ in $\Omega_0$ and $\nu\times y_1=0$ in $\partial\Omega_0$. (Note that the argument works also in the case that $\Omega_>$ is empty.) Hence, 
$y_1|_{\Omega_0} \in H_0(\mathrm{curl}\,0,\Omega_0)\cap H(\mathrm{div}_\ve 0,\Omega_0)$. The dimensions of this space is equal to the number of components of its boundary \cite[Proposition 6.1]{AssCiaLab}. 
So far we have proved 
\[
\ker B \subset \big\{(y_1,0,0)\colon y_1|_{\Omega_>}=0 \mbox{ and } y_1|_{\Omega_0} \in H_0(\mathrm{curl}\,0,\Omega_0)\cap H(\mathrm{div}_\ve 0,\Omega_0)\big\}\,.
 \] 
However, all element in the set on the right-hand side are in the kernel of $B$, 
which establishes \eqref{e:kernelB}.  
\end{proof}

\smallskip
\noindent
{\em Proof of Theorem~\ref{t:strong-stability}.}
It suffices to combine Theorem~\ref{t:convergence} with Proposition~\ref{p:last}.
Indeed, it is apparent that under some additional hypotheses (than the ones of Theorem~\ref{t:convergence}) the equilibrium solutions of the dynamical system $(Y,S(t))$ span a finite-dimensional set, whose dimension depends on the topological properties of
$\Omega_0$. 
In particular in the case that $\partial \Omega_0$ consists only of one component, the semigroup $S(t)$ is strongly stable.   
\qed

\begin{remark}
Under the same assumptions as in the last proposition we have
\begin{equation*}
\ker B =\big\{(\nabla\phi,0,0)\in H\colon \phi \in H^1_0(\Omega), \; 
\mathrm{div}(\ve \nabla\phi)=0 \mbox{ in } \Omega_0, \; \phi = c \mbox{ in } \Omega_>\big\}\,.
\end{equation*}
In the case that $\Omega_>$ consists of more than one component, the constant value can be different for different components.
\end{remark}

{\em Proof of Theorem~\ref{t:strong-stability2}.}
It suffices to combine Theorem~\ref{t:convergence} the with remark above. Let $z\in \ker B$. Then 
\[
 \langle z,y^0\rangle_H = 
 \int_\Omega \langle \nabla \phi, \ve e^0 \rangle\,dx =0,
\]
which shows that $y^0 \in (\ker B)^\perp = \overline{R(B)}$. 
Now we apply Theorem \ref{t:convergence}, or more precisely its proof, to obtain 
$\|S(t)y^0\|_H \to 0$ as $t\to +\infty$.
\qed

% ACKNOWLEDGEMENTS

\section*{Acknowledgements}
This research was started while M.~Eller was the guest of F.~Bucci at the Dipartimento di Matematica e Informatica of the Universit\`a degli Studi di Firenze (Unifi). 
Eller's stay was partially supported within the ``Internazionalizzazione'' programme of Unifi, which both authors gratefully acknowledge. 
M. Eller also acknowledges support of the Department of Mathematics and Statistics of Georgetown University. 
\\
F.~Bucci is a member of the Gruppo Nazionale per l'Analisi Mate\-ma\-tica, la Probabilit\`a  e le loro Applicazioni (GNAMPA) of the Istituto Nazionale di Alta Matematica; %(INdAM)
she is responsible of the 2026 GNAMPA research project ``Partial differential equations methods for control and inverse problems''.

% APPENDIH

\appendix %This command ends the counting of sections.

\section{Additional remarks on the constitutive model}

In this appendix we collect several remarks concerning the constitutive
assumptions introduced in Section~\ref{ss:derivation}, together with
examples of materials exhibiting thermoelectric or pyroelectric
behavior compatible with the model under consideration.

\subsection{Anisotropy and thermal polarization}

The constitutive law
\[
D=\varepsilon e+a\vartheta
\]
contains a temperature-dependent polarization contribution, where
\[
\vartheta=\theta-\theta_0
\]
denotes the temperature variation around a reference equilibrium state.

The vector $a\in\mathbb R^3$ represents a thermal polarization
direction and models thermoelectric or pyroelectric effects in
anisotropic polar media. The permittivity tensor
$\varepsilon$ is therefore allowed to be anisotropic, namely, a
symmetric positive definite matrix-valued field.

In crystalline materials, the orientation of $a$ is typically related
to preferred polarization directions determined by the symmetry class
of the medium. The constitutive relation
\[
D=\varepsilon e+a\vartheta
\]
should therefore be viewed as a linearized phenomenological model for
temperature-induced polarization effects near a reference equilibrium
configuration.

The coefficient $c$ appearing in
\[
  s
  =
  \theta_0\langle \varepsilon^{-1}a,D_t\rangle
  +
  c\vartheta_t .
\]
is interpreted as an effective heat-capacity parameter. For passive
thermodynamically stable materials one assumes $c>0$.

\subsection{Relation with Joule-heating models}

A more classical thermal effect in electromagnetic media is described
by Joule heating. In that setting the conductivity depends on the
temperature and the heat equation contains a forcing term of the form
\[
\sigma |e|^2 .
\]
The resulting Maxwell--heat system is nonlinear and its well-posedness
theory is substantially more delicate; see
\cite{Yin98,Yin06}. A linearized model was considered in \cite{BF15}.

In contrast, the present model introduces a linear thermoelectric
coupling through the constitutive law
\[
D=\varepsilon e+a\vartheta,
\]
leading to the coupled system \eqref{e:PDE-system}.

\subsection{Examples of admissible materials}

The model is intended for anisotropic dielectric media exhibiting
thermally induced polarization effects. Examples include ferroelectric
or pyroelectric materials such as barium titanate (BaTiO$_3$),
lithium niobate (LiNbO$_3$), lead zirconate titanate (PZT), and
tourmaline, as well as engineered polymer composites with aligned
microstructures.

Barium titanate and lithium niobate are classical examples of
anisotropic materials displaying strong pyroelectric response, while
PZT ceramics exhibit coupled electrical, thermal, and mechanical
effects after poling. Similar anisotropic thermoelectric behavior may
also arise in engineered polymer composites based on aligned fibers or
nanostructures.

The present constitutive framework should be viewed as a linearized
phenomenological approximation near a reference equilibrium state and
is not intended to provide a quantitative microscopic description of
specific materials.

\subsection{Related literature}

Constitutive models involving anisotropic thermoelectric or
pyroelectric couplings have been considered in several mathematical
and physical settings.

Direction-dependent thermoelectric transport properties in anisotropic
semiconductor compounds are investigated in \cite{Guo2015}. Coupled
thermal-electromagnetic effects in anisotropic pyroelectric media are
studied in \cite{Abdalla2016,Alshaikh2014}, while effective
thermoelectric coefficients in anisotropic polycrystalline materials
are discussed in \cite{Basaula2022}.

Additional background on constitutive modeling in electromagnetic
continuous media may be found in \cite{FabMor}.

% REFERENCES

% THE END
\end{document}